\documentclass{amsart}

\usepackage{amssymb}
\usepackage{amsfonts}

\usepackage{graphicx}
\usepackage{epstopdf}

\newcommand{\Spalt}          {\ensuremath{\mathcal{S}}}
\newcommand\compG{\widehat{\mathcal{R}}{}^G}
\newcommand\compR{\widehat{\mathcal{R}}{}^R}

\newcommand{\ie}          {\mbox{, \emph{i.\,e.},}}
\newcommand{\eqv}         {\ensuremath{\equiv}}
\newcommand{\eqvelem}     {\eqv^{(1)}}
\newcommand{\lb}          {\langle}
\newcommand{\rb}          {\rangle}
\newcommand{\Rp}          {\ensuremath{\mathcal{R}}}
\newcommand{\Bgb}         {\ensuremath{\mathcal{B}}}
\newcommand{\ew}          {\ensuremath{\varepsilon}}
\newcommand{\pres}        {presentation}
\newcommand{\press}       {presentations}
\newcommand{\sr}          {\ensuremath{(\Spalt;\Rp)}}
\newcommand{\monoid}         {\ensuremath{\lb\Spalt;\Rp\rb^+}}

\newcommand{\G}           {Gr\"obner}

\newcommand{\gen}         {generator}
\newcommand{\gens}        {{\gen}s}
\newcommand{\red}         {\ensuremath{\leadsto}}
\newcommand{\GR}          {type~2}
\newcommand{\RG}         {type~1}
\newcommand{\mixed}       {type~3}
\newcommand{\cex}         {counter-example}
\newcommand{\cexs}        {counter-examples}
\newcommand{\rev}        {\ensuremath{\curvearrowright}}

\newcommand{\revelem}    {\ensuremath{\rev^{(1)}}}

\newcommand{\redr}	{reversing}

\newtheorem{theo}{Theorem}[section]
\newtheorem{lemma}[theo]{Lemma}
\newtheorem{prop}[theo]{Proposition}
\newtheorem{algo}[theo]{Algorithm}
\newtheorem{algo-defi}[theo]{Algorithm-Definition}
\newtheorem{algo-prop}[theo]{Algorithm-Proposition}

\newtheorem{fact}[theo]{Fact}
\newtheorem{ques}[theo]{Question}
\newtheorem*{props}{Proposition}

\theoremstyle{definition}
\newtheorem{defi}[theo]{Definition}
\newtheorem{example}[theo]{Example}

\theoremstyle{remark}
\newtheorem{rem}[theo]{Remark}

\numberwithin{equation}{section}
\makeatletter
\renewcommand{\fnum@figure}{\small{\figurename~\thefigure}}
\makeatother

\begin{document}

\title{Comparing Gr\"obner bases and word reversing}

\author{Marc Autord}
\address{}
\curraddr{}
\email{autord@math.unicaen.fr}


\begin{abstract}
Gr\"obner bases, in their noncommutative version, and word reversing are methods
for solving the word problem of a presented monoid, and both rely on iteratively
completing the initial list of relations. Simple examples may suggest to
conjecture that both completion procedures are closely related. Here we
disprove this conjecture by exhibiting families of presentations for which they
radically differ.
\end{abstract}

\maketitle

\section*{Introduction}

When an algebraic structure is given by \gens\ and relations, typically a semigroup or a group, each element admits in
general several word representatives. The word problem is the question of deciding whether two words represent the same
element. It is known that, both in case of semigroups and groups, the word problem can be undecidable
\cite{iii71:_group_theor_decis_their_class}. However a number of methods have been developed that solve the word
problem in good cases. The aim of this paper is to compare two such methods,
namely the well-known \G\ basis method
\cite{kostrikin95:_combin_asymp_method_algeb} as adapted to semigroups, and the word \redr\ method of
\cite{dehornoy92:_deux}. 
Originally designed to answer questions in the context of free commutative algebras, the method of \G\ bases has
subsequently been adapted to noncommutative algebras and, \emph{via} the inclusion of a semigroup~$G$ in the
algebra~$K\langle S\rangle$, to general semigroups. In the latter case, the method consists in starting with a
semigroup presentation~$\sr$ and in running a certain completion procedure that adds new relations that are consequences
of the initial ones until one possibly obtains a so-called reduced \G\ basis \cite{mora86:_groeb, bokut03:_shirs} ---
see~\cite{kostrikin95:_combin_asymp_method_algeb} for a survey.

Word \redr\ is another combinatorial method \cite{dehornoy92:_deux} for investigating presented
semigroups. It also consists in iterating some simple syntactic transformation on words. In good cases, the method can be used to solve
the word problem. However, this only happens when the initial presentation satisfies a certain completeness condition. When
this is not the case, there exists a completion procedure that, as in the case of the \G\ completion, consists in adding new
relations to the initial presentation~\cite{dehornoy03:_compl}.

We thus have two {\it a priori} unrelated completion procedures. Now, it can be observed on some simple examples that the
two processes lead to adding the same relations. It is therefore natural to address the question of how the
two completions are connected. The aim of this paper is to show that, actually, there is no simple general relation:

\begin{props}
There exist finite semigroup presentations for which the \G\ completion and the \redr\ completion disagree.
\end{props}

Actually, we shall prove a more precise statement---Proposition~\ref{prop:main} below---showing really independent
behaviours, namely examples where one completion is finite and the other is infinite, or where there is an inclusion or
no inclusion. 

Thus the paper is mainly composed of negative results and counter-examples. However, we think it is interesting to list
these many examples because neither the \G\ nor the \redr\ completion procedures have been much investigated so
far, and their global behaviour is not well understood. Also, we point out that most of the examples we investigate
below are not just artificial {\it ad hoc} constructions, but they involve well-known and interesting semigroups, in
particular the braid semigroups and the Heisenberg semigroup: so, in particular, our results give explicit \G\ bases for these cases.

The paper is organized as follows. In Section~\ref{sec:grobner_bases}, we recall the notion of a \G\ basis in the  context of
presented monoids. In Section~\ref{sec:word_reversing}, we similarly describe word \redr\ and its associated notion of
completion and observe that the latter coincides with the \G\ completion of Section~\ref{sec:grobner_bases} on
simple examples. Then, in Section~\ref{sec:complet_diverg}, we establish the main results by constructing explicit
counter-examples witnessing all possible types of divergence for the two completions. In Section~\ref{sec:cancellativity}, 
we quit the word problem and address another problem involving presented semigroups, namely recognizing
cancellativity, and we examine its possible connections with \G\ bases. 

\bigskip

\noindent {\bf Acknoledgments. } The author is greatly indebted to Patrick~Dehornoy for suggesting the problem and for his many helpful comments. 

\section{Gr\"obner bases in the framework of presented monoids}
\label{sec:grobner_bases}

Both in the commutative and noncommutative cases, \G\ bases have been originally designed to decide whether a 
polynomial belongs to a given ideal of some algebra $K[X_1,\ldots,X_n]$ or $K\lb X_1,\ldots,X_n\rb$.  It is however easy
to adapt the \G\ bases machinery so as to solve the word problem of presented semigroups. 

In this section, we briefly review the method, in the specific case of semigroups. In particular, we recall how, in that case,
the \G\ completion procedure can be entirely performed in the language of semigroups, and how \G\ bases can be used
to solve the word problem.

\subsection{Gr\"obner bases}
\label{s_sec:grobner_bases}

Following standard notation, if \Spalt\ is a nonempty set, $\Spalt^*$ denotes the free monoid generated by \Spalt\ie\  the set of all words on \Spalt\ together with concatenation. The empty word is denoted by $\ew$. In 
the sequel, \Rp\ is a set of pairs of
nonempty words on \Spalt. A pair \sr\ is called a \emph{semigroup \pres} and we call $\Spalt$ the set of \emph{\gens}
and $\Rp$ the set of \emph{relations}. In the sequel, it will be convenient---in particular for Section~\ref{sec:word_reversing}---to
work with monoids rather than semigroups, that is, to add a unit. We then define
$\lb\Spalt;\,
\Rp
\rb^+$ to be the monoid
$\Spalt^*/\eqv$, where
$\eqv$ is the smallest congruence on $\Spalt^*$ containing \Rp. We denote $\eqvelem$ the relation so defined: ``$w
\eqvelem w'$ holds if and only if going from $w$ to $w'$ can be done by applying one relation of \Rp''. 

For a field $K$, the free associative $K$-algebra (or simply free algebra) with set of \gens~\Spalt\ and unity is denoted by
$K\lb\Spalt\rb$. As a set it can be viewed as the set of all linear combinations of elements of $\Spalt^*$ with coefficients
in $K$. The free monoid~$\Spalt^*$ embeds in the free algebra $K\lb\Spalt\rb$, and, more generally, the monoid
$\lb\Spalt;\,\Rp\rb^+$ embeds in the factor algebra $K\lb\Spalt\rb/I$, where $I$ is the two-sided ideal generated by all
the elements $u-v$, with $(u,v)\in \Rp$.

For details about \G\ bases in the context of algebras we refer to \cite{kostrikin95:_combin_asymp_method_algeb},  of
which we follow the notation and the terminology. We fix a well-ordering $<$ on the set of words~$\Spalt^*$, that is,  any
two words are comparable and every nonempty subset has a smallest element. This enables us to perform inductive proofs
on the rank of words. Moreover, we assume that the ordering is compatible with the operation of the monoid: $f < g$
implies $ufv < ugv$ for all words $u$, $v$. Such an ordering is called \emph{admissible}. There always exists an
admissible ordering: for each linear ordering of \Spalt\ the associated \emph{deglex} ordering (or homogeneous
lexicographic ordering) satisfies all these conditions: the words are first ordered by their length, and, if the lengths are the
same, lexicographically. In the sequel, we shall often write $u=v$ instead of $(u,v)$ for relations of a presentation.

Adapting \G\ techniques to the context of monoids---or semigroups---is easy, and it is alluded to in~\cite{mora86:_groeb}, but it
seems not to have been written explicitly in literature, and, therefore, we include some details.  The next proposition is the
first step: equalities in a presented monoid \monoid\  correspond to equalities of monomials in the algebra
$K\lb\Spalt\rb/I$, where $I$ is an ideal determined by~$\Rp$.

\begin{prop}
\label{prop:eqv_appartenir-ideal_eqv-monoid}
Assume $(\Spalt;\, \Rp)$ is a semigroup \pres. Let $I$ be the two-sided ideal of the free algebra $K\lb \Spalt \rb$ generated by the elements $u-v$ with $(u,v)\in\Rp$. Then, for all words $w$, $w'$ on \Spalt, the following are equivalent:
\begin{tabbing}
  \qquad \= $(i)$  \quad \= $w\eqv w'$,\\
         \> $(ii)$       \> $w-w' \in I$.
\end{tabbing}
\end{prop}

\begin{proof}
Suppose $w\eqv w'$. This means that there exists a finite sequence of words $w_0=w,\ldots,w_n=w'$ such that $w_i \eqvelem w_{i+1}$ holds for every $i<n$. We prove by induction on $n$ that $w-w'$ lies in $I$. Assume $n=1$; there exist words $t,u$ on \Spalt\ and a relation $v=v'$ of \Rp\ such that both $w=tvu$ and $w'=tv'u$ hold. By hypothesis, $(v,v')\in\Rp$ implies $v-v'\in I$ and hence $t(v-v')u\in I$, that is $w-w'\in I$. Assume $n>1$. Then we have 
\[
w' = w_0 \eqvelem w_1 \eqv w_n =w.
\]
By induction hypothesis we get $w_n-w_1\in I$ and $w_1-w_0\in I$, and, writing $w-w'=(w_n - w_1) + (w_1-w_0)$ 
shows that $w-w'$ lies in $I$.

Suppose $w-w'\in I$. First observe that there is a decomposition $w-w'=\sum_{i=1}^{n} t_{i} (v_{i}-v'_{i}) u_{i}$  for some
$t_i,u_i\in K\lb\Spalt\rb$ and $(v_i,v'_i)\in\Rp$; this decomposition implies that there is a relation $v=v'$ such that $v$
is a subword of $w$, say $w=tvu$, with $t,u\in\Spalt^\ast$. Therefore $w-w'\in I$ implies $w-w'-(tvu-tv'u)\in I$ and then
$tv'u-w'\in I$. Suppose $w>w'$ and that $w$ is the smallest element for which the conclusion does not hold. Thus we get
$tv'u\eqv w'$, hence $w=tvu\eqvelem tv'u \eqv w'$.
\end{proof}

With Proposition~\ref{prop:eqv_appartenir-ideal_eqv-monoid}, we established a connection between words being 
equivalent and their difference lying in a particular ideal. The next lemma shows that for such ideals, that is to say ideals
generated by differences of monomials, the elements of the reduced \G\ basis are again differences of monomials.

\begin{lemma}
\label{lemma:groebner_reste_monoid}
  Assume $I$ is the ideal of $K\lb\Spalt\rb$ generated by $u_1-v_1,\ldots,u_n-v_n$ with $u_i$, $v_i$ in $\Spalt^*$. Then all the elements added during the G-completion have the type $u-v$, with $u$ and $v$ in $\Spalt^*$.
\end{lemma}

\begin{proof}
  Three steps are involved in the algorithm, namely normalization, reduction and completion. 

Define $\mathcal{B}$ to be the set of the elements $u_i-v_i$, $1\leq i\leq n$. The normalization process consists in substituting every element of $\mathcal{B}$ with a proportional element with leading coefficient $1$. In our context, this step does not change any of the elements of $\mathcal{B}$.

In the sequel, if $p$ is a polynomial, $\hat{p}$ denotes the term of highest rank and $\check{p}$ denotes $p-\hat{p}$. Assume $t\in\mathcal{B}$. Suppose there is a $u$ in $\mathcal{B}$ such that $\hat{u}$ is a subword of $\hat{t}$; in other words, there are words $t_l$ and $t_r$ satisfying $\hat{t}=t_l\hat{u}t_r$. Then the reduction step consists in discarding $t$ from $\mathcal{B}$ and replacing it by $t-t_lut_r$. We check now that $t-t_lut_r$ has the prescribed type: 
\[
  t-t_lut_r  =  (\hat{t}-t_lut_r) -\check{t} +t_l\check{u}t_r = t_l\check{u}t_r-\check{t}.
\]

The third step is composition, which forces leading terms to collapse when appropriately combined. Given $u=xy-\check{u}$ and $v=yz-\check{v}$, the composition is $uz-xv$, in which the two leading terms $xy$ and $yz$ cancel. Now we have $uz-xv=x\check{v}-\check{u}z$, again a difference of two monomials.
\end{proof}

Thus, along the G-completion\ie\ the computation of the \G\ basis, the elements added to the basis are differences of monomials. By Proposition~\ref{prop:eqv_appartenir-ideal_eqv-monoid}, they correspond to equalities in the monoid $\lb\Spalt;\, \Rp\rb^+$ and therefore to equivalent words on \Spalt. This allows us to redefine the \G\ operations at the level of words directly:

\begin{defi}[reduction of a relation]
\label{def:reduction_relation}
Let \sr\ be a semigroup \pres. Let $(w,w')$ and $(v,v')$ be relations satisfying $v>v'$ and $w=v_lvv_d$, with $v_l,v_d\in\Spalt^\ast$. Then the relation $(v_lv'v_r,w')$ is said to be obtained by \emph{reduction} of $(w,w')$ by $(v,v')$. We say that the relation $(w,w')$ \emph{reduces to 0} with respect to the set $\Rp$, or simply, when the set of relations is clear from the context, that $(w,w')$ reduces to 0, if there is a sequence of relations $(w,w')=(w_0,w_0'),\ldots, (w_n,w'_n)=(u,u)$ of  $\eqv$ such that every $(w_{i+1},w_{i+1}')$ is a reduction of $(w_{i},w_{i}')$ by a relation $(u_i, u'_i)$ of \Rp.
\end{defi}

Note that a reduction of $w=w'$ by any relation yields a relation $z=z'$ that satisfies $z<w$ or $z'<w'$. Since $\Spalt^*$ is well-ordered, reducing iteratively a relation eventually stops; otherwise we would get an infinite decreasing sequence of words. When no reduction applies to a relation, it is said to be \emph{reduced}.

\begin{defi}[composition of relations]
\label{def:compo_relation}
  Let \sr\ be a semigroup \pres. Let $w=w'$ and $v=v'$ be relations in $\Rp$ such that $w$ and $v$ overlap\ie\  we have $w=xy$ and $v=yz$ with $y$ a nonempty word. The \emph{composition} of $w=w'$ and $v=v'$ with overlapping $y$ is the element $(xv', w'z)$ of $\Spalt^\ast\times\Spalt^\ast$.
\end{defi}

\begin{fact}
  The composition of two elements of $\eqv$ is again in $\eqv$.
\end{fact}

\begin{rem}
  To compute the composition of $(xy,w')$ and $(yz,v')$ with overlapping $y$, we write
\[
(xy-w')z-x(yz-v') = xv'-w'z
\]
and we deduce that the composition is $(xv', w'z)$. This convenient way of composing relations is in fact the way of composing polynomials in algebras.
\end{rem}

Adapting the classical notions---see for instance~\cite{kostrikin95:_combin_asymp_method_algeb}---we introduce:

\begin{defi}[\G\ basis and G-completeness]
\label{def:grob_basis}
  Let \sr\ be a semigroup \pres. A subset $\mathcal{B}$ of the congruence $\eqv$ generated by \Rp\ is a \emph{\G\ basis} of \sr\ if every element $(u,v)$ of $\mathcal{B}$ satisfies $u>v$ and if, for any two equivalent words $w$, $w'$ in $\Spalt^*$ with $w>w'$, there exists an element $(u,v)$ in $\mathcal{B}$ such that $u$ is a subword of $w$. If \Rp\ is a \G\ basis, the \pres\ is said to be \emph{G-complete.}
\end{defi}

Not every \pres\ is G-complete: in the monoid associated to the \pres\ $(a,b; abababa=bb)$ with homogeneous lexicographic order, the words $b^3a$ and $ab^3$ are equivalent and yet none of them admits $abababa$ as a subword.

Nevertheless, \G\ bases do exist: starting from \sr, complete the set of relations with every equality $u=v$ that holds in \monoid. The set of relations obtained this way is a \G\ basis but there is no practical interest of such a completion as it is noneffective: it requires a former solution to the word problem. Moreover, there is redundancy in such a basis in the sense that if $u=v$ is a relation, then $wu=wv$ is also a relation and both appear in the basis, although $wu=wv$ can be reduced by $u=v$. We shall see in the sequel, however, that running Algorithm~\ref{algo:grob_completion_monoid} completes the set of relations into a smaller \G\ basis than the set of all relations, that it does not need a solution to the word problem and that no redundancy is left in the \G\ basis so obtained\ie\ the \G\ basis is reduced.

\begin{defi}
\label{def:reduced_minimal_grob_basis}
  A \G\ basis is \emph{minimal} if no subset of it is a \G\ basis. A set of relations $\mathcal{U}$ is \emph{reduced} if every relation of the \pres\ is reduced. 
\end{defi}

\begin{rem}
\label{rem:reduced_implies_not_same_leading_words}
  A set $\mathcal{U}$ of relations in which there exist two relations involving the same leading word is not reduced: if $u=v$ and $u=w$ are two relations of $\mathcal{U}$ satisfying $v<w$, one reduces $u=v$ by $u=w$ to $w=v$ and then $u=w$ to $u=v$.
\end{rem}

To recognize whether a set of relations is a reduced \G\ basis, we have the following criterion:

\begin{prop}
\label{prop:minimal_grob_if_no_red_and_no_comp_left}
  Assume $(\Spalt;\, \Rp)$ is a semigroup \pres\ and $\mathcal{U}$ a set satisfying $\Rp \subseteq \mathcal{U}\subseteq \eqv$. If $\mathcal{U}$  is reduced and if every composition of two relations of $\mathcal{U}$ reduces to 0, then $\mathcal{U}$ is a reduced \G\ basis of $\sr$. 
\end{prop}

\begin{proof}
The result is a rewriting of Lemma on Composition~\cite[p.~30]{kostrikin95:_combin_asymp_method_algeb} adapted to the context of monoids using Proposition~\ref{prop:eqv_appartenir-ideal_eqv-monoid}.
\end{proof}

The conjunction of Proposition~\ref{prop:eqv_appartenir-ideal_eqv-monoid} and Lemma~\ref{lemma:groebner_reste_monoid} give an algorithm (based on the one working in the free algebras) that computes a reduced \G\ basis for a semigroup \pres. There are no references to algebras nor to ideals either, the whole process taking place in the monoid. In a reduced set of relations, we order pairs by putting 
\[
(u_1=v_1,u_2=v_2)<(u_3=v_3,u_4=v_4) \Leftrightarrow u_1u_2 < u_3u_4.
\]
Since the set of relations is reduced, the order $<$ is linear.

\begin{algo}[G-completion]
\label{algo:grob_completion_monoid}
Assume $(\Spalt;\,\Rp)$ is a \pres. \\
Start with $\mathcal{U}=\Rp$.
\begin{tabbing}
\quad \= 1. \= Reduce all the relations of $\mathcal{U}$ until no possible reduction is left;\\
      \> 2. \> Delete all relations $v=v$ of $\mathcal{U}$; \\
      \> 3. \> $\mathtt{WHI}$\=$\mathtt{LE}$ there exist two relations of $\mathcal{U}$ that can be composed \\
      \>    \>    \> Add the result of composition of the smallest pair to $\mathcal{U}$;\\
      \>    \>    \> Go to 1;\\
      \> $\mathtt{OUTPUT:}$ $\mathcal{U}$.
\end{tabbing}
\end{algo}

\begin{prop}
\label{prop:grob_completion_monoid}
  If Algorithm~\ref{algo:grob_completion_monoid} terminates, the final set of relations $\mathcal{U}$ is a
reduced \G\ basis of \sr.
\end{prop}

\begin{proof}
The set of relations $\mathcal{U}$ eventually obtained fulfills the conditions of Proposition~\ref{prop:minimal_grob_if_no_red_and_no_comp_left}.
\end{proof}

Algorithm~\ref{algo:grob_completion_monoid} need not terminate in a finite number of steps. In fact, as we will see in Section~\ref{s_sec:using_grobner_bases}, whenever G-completion stops, we can solve the word problem. Conversely, if \sr\ is a \pres\ for a semigroup with undecidable word problem---and such \press\ exist---then the execution of Algorithm~ \ref{algo:grob_completion_monoid} on \sr\ cannot terminate. But we shall see below that this may also happen for \press\ of semigroups with an easy word problem.

\subsection{An example}
\label{s_sec:grobner_example}

We now illustrate Algorithm~\ref{algo:grob_completion_monoid} by computing a reduced \G\ basis explicitly. This example also shows that Algorithm~\ref{algo:grob_completion_monoid} may not terminate.

\begin{example}
\label{eg:meme_completion_cas_groebner}
Consider the \pres\ ${(a,b;\, bab=ba^2)}$ and the deglex ordering induced by $b>a$. Following Algorithm~\ref{algo:grob_completion_monoid}, we alternatively perform reduction steps and composition steps. A reduction by the relation numbered $(i)$ will be denoted $\stackrel{(i)}{\red}$. We start with $\mathcal{U}=\{ bab=ba^2\}$ and we number $(1)$ the relation $bab=ba^2$. 

Since there is no reduction at this stage, we first compose $(1)$ with itself to get:
\[
  (bab-ba^2)ab-ba(bab-ba^2)  =  -ba^2ab+bab^2  \stackrel{(1)}{\red} -ba^3b +a^4b,
\]
that is, the relation 
\begin{equation}
  \label{eq:ba^3b=a^4b}
  \tag{2}
  ba^3b=a^4b.
\end{equation}
Thus we obtain $\mathcal{U}_1=\{bab=ba^2,\, ba^3b=ba^4\}$. There is no reduction since $bab$ is not a subword of $ba^3b$. Composing $(2)$ with $(1)$, we get:
\[
  (ba^3b-ba^4)ab-ba^3(bab-ba^2) =  -ba^4ab+ba^3ba^2 \stackrel{(2)}{\red}  -ba^5b +a^6b,
\]
and therefore the relation
\begin{equation}
  \label{eq:ba^5b=a^6b}
  ba^5b=a^6b.
\tag{3}
\end{equation}
We obtain $\mathcal{U}_2=\{ bab=ba^2,\, ba^3b=ba^4,\, ba^5b=ba^6\}$. We claim that the algorithm successively adds all relations 
\begin{equation}
  \label{eq:ba^2n-1b=a^2nb}
  ba^{2n-1}b=a^{2n}b.
  \tag{n}
\end{equation}
We prove it by induction on $n$. The case $n=1$ is clear. Suppose $n>1$ and compose relation $(n)$ with relation $(1)$ to get $(n+1)$:
\begin{eqnarray*}
&&  (ba^{2n-1}b-ba^{2n})ab-ba^{2n-1}(bab-ba^2) \\
&&\hspace{4cm} =  -ba^{2n}ab+ba^{2n+1}ba^2\\
&&\hspace{4cm} \stackrel{(n)}{\red} -ba^{2(n+1)-1}b +ba^{2n+2}\\
&&\hspace{4cm} =  -ba^{2(n+1)-1}b +ba^{2(n+1)}.
\end{eqnarray*}
Hence we obtain $\mathcal{U}_\infty=\{ ba^{2n-1}b=ba^{2n};\, n\geq 1 \}$. We claim that $\mathcal{U}_\infty$ is a reduced \G\ basis of the \pres\ $(a,b;\, bab=ba^2)$. By Proposition~\ref{prop:minimal_grob_if_no_red_and_no_comp_left}, it suffices to check that all compositions in $\mathcal{U}_\infty$ reduce to 0. Now the composition of $(n)$ and $(m)$ is:
\begin{eqnarray*}
&& (ba^{2n-1}b-ba^{2n})a^{2m-1}b - ba^{2n-1}(ba^{2m-1}b-ba^{2m}) \\
&&\hspace{6cm} = -ba^{2n}a^{2m-1}b + ba^{2n-1}ba^{2m}\\
&&\hspace{6cm} \stackrel{(n)}{\red} -ba^{2(m+n)-1}b +a^{2(m+n)}b \\
&&\hspace{5.767cm} \stackrel{(m+n)}{\red} -ba^{2(m+n)} +ba^{2(m+n)} = 0.
\end{eqnarray*}
So there is no reduction and no composition left in $\mathcal{U}_\infty$. Thus, the set $\mathcal{U}_\infty$ is a reduced \G\ basis.
\end{example}

\subsection{Using Gr\"obner bases}
\label{s_sec:using_grobner_bases}

In this section, we recall that the knowledge of a \G\ basis of a semigroup \pres\ allows to solve the word problem of the associated monoid.

\begin{defi}
  Let \sr\ be a G-complete semigroup \pres. A word $u$ of $\Spalt^\ast$ is \emph{G-reduced} (or simply \emph{reduced}) if none of its subwords appears in a relation of \Rp\ as a leading word. The \emph{G-reduction} of a word $u$ of $\Spalt^\ast$ is the G-reduced word $\underline{u}$ \Rp-equivalent to $u$.  
\end{defi}

\begin{rem}
  The unicity of G-reduction follows from the properties of \G\ bases: let $\underline{w}$ and $\underline{w'}$ be two G-reductions of a word $u$; the equivalence $\underline{w}\eqv \underline{w'}$ implies that there exists a relation $v=v'$ in \Rp, with $v$ a subword of $max(\underline{w},\underline{w'})$, contradicting the hypothesis that both words were reduced.
\end{rem}

G-reduction provides a unique normal form for each element of the considered monoid, and therefore solves the word
problem: 

\begin{prop}
Assume that $\sr$ is a G-complete presentation. Then two words~$w, w'$ of~$\Spalt^*$
represent the same element of \monoid\ if and only if the reductions of~$w$ and~$w'$ are equal.
\end{prop}

\begin{example}
With the setting of Section~\ref{s_sec:grobner_example}, the word $aba^3$ is G-reduced; the word $aba^3bab$ is not reduced however, since both $ba^3b$ and $bab$ appear in $\mathcal{U}_\infty$ as leading words of relations, namely $ba^3b=ba^4$ and $bab=ba^2$. To reduce $aba^3bab$, we substitute, for example, the subword $ba^3b$ with $ba^4$ to get $aba^5b$ and then $ba^5b$ with $ba^6$ to obtain the G-reduced word $aba^6$ equivalent to $aba^3bab$. Starting the reduction with the relation $bab=ba^2$ instead yields $aba^3ba^2$ and then reducing with the relation $ba^3b=ba^4$, we get again the reduced word $aba^6$.
\end{example}

\section{Word \redr}
\label{sec:word_reversing}
Word \redr\ is a combinatorial operation on words that also solves the word problem for a presented monoid whenever the considered \pres\ satisfies an \emph{ad hoc} condition called completeness. Not every \pres\ is complete for \redr, but, as in the case of \G\ bases, there exists a completion procedure that possibly transforms an initially incomplete \pres\  into a complete one. 

\subsection{Word \redr}
\label{s_sec:word_reversing}

We recall results about word \redr\ (and refer to~\cite{dehornoy03:_compl} for more details) so as to be able to compare this technique with the \G\ methods exposed in Section~\ref{s_sec:grobner_bases}. 

Let \sr\ be a semigroup \pres. For every letter $s$ in \Spalt, we introduce a disjoint copy $s^{-1}$ of $s$ and we denote by $\Spalt^{-1}$ the set of all $s^{-1}$'s. The elements of \Spalt\ (resp. $\Spalt^{-1}$) are said to be \emph{positive} (resp. \emph{negative}). For $s_1,\ldots,s_n\in\Spalt$ and $u=s_1\ldots s_n$ a word in $\Spalt^\ast$, we write $u^{-1}$ for the word $s_n^{-1}\ldots s_1^{-1}$ in $\Spalt^{-1\ast}$.

\begin{defi}[\redr]
\label{defi:word_reversing} 
Let \sr\ be a semigroup \pres\ and let~$w$ and~$w'$ be words on~$\Spalt\cup\Spalt^{-1}$. We say that~$w\revelem w'$ is true if~$w'$ is obtained from~$w$
\begin{itemize}
\item[--] either by deleting a subword $u^{-1}u$ with $u\in\Spalt^+$,
\item[--] or by replacing a subword $u^{-1}v$ where $u$, $v$ are nonempty words on \Spalt\ with a word
$v'u'{}^{-1}$ such that $uv'=vu'$ is a relation of \Rp. 
\end{itemize}
We say that $w$ is \emph{reversible} to $w'$, and we write $w\rev w'$, if there exists a sequence of words $w_0$, $w_1$, $\ldots$, $w_n$ satisfying $w_i\revelem w_{i+1}$ for every $i$ and $w=w_0$ and $w'=w_n$. We say in that case that $w'$ is a \emph{\redr} of $w$.
\end{defi}

\begin{example}
Take the standard Artin \pres\ $(a,b; bab=aba)$ and start with the word $a^{-1}b^2$: we successively get
\[ a^{-1}b^2 \rev\ bab^{-1}a^{-1}b \rev\ bab^{-1}bab^{-1}a^{-1} \rev\ ba\ew ab^{-1}a^{-1}= ba ab^{-1}a^{-1}.\]
Note that \redr\ sequences need not terminate. We say that word \redr\ is  \emph{convergent} if, starting from any word, there exists a terminating \redr\ sequence. The \pres\ above is convergent, whereas  $(a,b;ba = a^2b)$ is not: 
\[ b^{-1}ab \rev\ ab^{-1}a^{-1}b \rev\ a\underline{b^{-1}ab}a^{-1}.\]
\end{example}

The next proposition exhibits a link between \rev\ and \eqv. 

\begin{prop}[\protect{\cite[Prop.~1.9]{dehornoy03:_compl}}]
\label{prop:ret_yields_monoid_equiv}
  Assume that \sr\ is a semigroup \pres, and $u$, $v$ are words in $\Spalt^\ast$. Then $u^{-1}v\rev \ew$ implies $u\eqv v$.
\end{prop}

The converse is not true in general: there exist \press\ for which word \redr\ fails to detect equivalence and thus does not solve the associated word problem; these \press\ lack the completeness property, which we define as follows.

\begin{defi}[R-completeness]
\label{def:word_revers_completeness}
A semigroup \pres\ is \emph{R-complete} if
$u\eqv\nobreak v$ implies $u^{-1}v\rev \ew$.
\end{defi}

By very definition, we have

\begin{prop}
Assume that $\sr$ is a R-complete presentation such that $\Rp$-word \redr\ is convergent. Then two words~$w, w'$
of~$\Spalt^*$ represent the same element of \monoid\ if and only if $w^{-1} w' \rev \ew$ holds.

\end{prop}

\label{s_sec:checking_completeness}

In the sequel, we shall need a criterion for establishing whether a presentation is possibly R-complete. We shall use the
one we describe now.

\begin{defi}[homogeneity]
\label{def:homogeneous}
  We say that a positive \pres\ \sr\ is \emph{homogeneous} if it admits a \emph{pseudolength}, the latter being defined as a map $\lambda$ of $\Spalt^\ast$ to the nonnegative integers, satisfying $\lambda(su)>\lambda(u)$ for each $s$ in \Spalt\ and $u$ in $\Spalt^\ast$, and invariant under $\eqv$.
\end{defi}

Note that if all pairs in \Rp\ have the same length then the length itself is a pseudolength for \sr. We state now a criterion to check R-completeness:

\begin{algo}
  \label{algo:word_rev_completeness_criterion}
Let \sr\ be a homogeneous semigroup \pres. For each triple of letters $s$, $t$, $r$ in \Spalt:
\begin{tabbing}
  \quad \= 1. \= Reverse $s^{-1}rr^{-1}t$ to all possible words of the form $uv^{-1}$, with $u,v\in\Spalt^\ast$;\\
        \> 2. \> For each $uv^{-1}$ so obtained, check $(su)^{-1}(tv)\rev \ew$.
\end{tabbing}
\end{algo}

\begin{prop}[\protect{\cite[Algorithm~4.8]{dehornoy03:_compl}}]
  \label{prop:word_rev_completeness_criterion}
Assume that \sr\ is a semigroup homogeneous \pres. Then \sr\ is R-complete if and only if the answer at Step 2 of 
Algorithm~\ref{algo:word_rev_completeness_criterion} is positive for each triple of letters $(r, s, t)$ and each word~$u
v^{-1}$ obtained at Step 1.
\end{prop}

\subsection{Reversing-completion}
\label{s_sec:completion_word_reversing}

When a \pres\ is not R-complete\ie\  when word \redr\ fails to prove some equivalence $u\eqv v$, there are completion procedures, in particular when the \pres\ is homogeneous:

\begin{algo}
\label{algo:word_rev_completion}
  The setting is the one of Algorithm~\ref{algo:word_rev_completeness_criterion}. 
\begin{tabbing}
  \quad \= $\mathtt{REPEAT}$ \=\\
        \>    \> Reverse $s^{-1}rr^{-1}t$ to all possible words of the form $uv^{-1}$;\\ 
        \>    \> $\mathtt{FOR}$ \= each $uv^{-1}$ so obtained:\\
        \>    \>     \> $\mathtt{IF}$ $(su)^{-1}(tv)\not\rev \ew$\\
        \>    \>     \> $\mathtt{THEN}$ add the relation $su=tv$ to the \pres;\\
        \> $\mathtt{UNTIL}$ no new relation has been added to the \pres;\\
        \> $\mathtt{OUTPUT:}$ a \pres.
\end{tabbing}
\end{algo}

\begin{prop}[\protect{\cite[\S~5]{dehornoy03:_compl}}]
\label{prop:word_rev_completion}
  When Algorithm~\ref{algo:word_rev_completion} terminates, it yields a R-complete \pres.
\end{prop}

Algorithm~\ref{algo:word_rev_completion} and Proposition~\ref{prop:word_rev_completion} are comparable to Algorithm~\ref{algo:grob_completion_monoid} and Proposition~\ref{prop:grob_completion_monoid} in the sense that if the considered \pres\ fails at a completeness test, namely $(su)^{-1}tv\rev \ew$ for word \redr\ and whether it exists compositions not reducing to 0 for \G, then both algorithms add the obstruction to the \pres.

\subsection{An example}
\label{s_sec:word-reversing_example}

In this section, we apply Algorithms~\ref{algo:word_rev_completeness_criterion} and~\ref{algo:word_rev_completion} to the example of Section~\ref{s_sec:grobner_example}. We first have to check whether the \pres\ $( a,b;\, bab=ba^2)$ is R-complete and then, if needed, R-complete it. 

In this case, the length of a word is invariant under $\eqv$ and is therefore a pseudolength; hence the \pres\ is homogeneous and we can apply Algorithm~\ref{algo:word_rev_completion}. There are eight triples of letters to deal with but as no relation of the \pres\ is of the type $a\ldots = b\ldots$, we are left with a single triple, namely $(b,b,b)$.

Before \redr\ $b^{-1}bb^{-1}b$, we introduce the notion of \redr\ graph and refer 
to~\cite{dehornoy99:_gauss_garsid_artin} for more details. A \redr\ graph is a directed and labelled graph that we
associate to a \redr\ sequence $w_0$, $w_1$,... as follows. First, we associate with $w_0$ a path labelled with the
successive letters of $w_0$: we associate to every positive letter $s$ an horizontal right-oriented  edge labelled $s$, and to
every negative letter $s^{-1}$ a vertical down-oriented edge labelled $s$. Then we successively represent the words $w_1$,
$w_2$,... as follows: if $w_{i+1}$ is obtained from $w_i$ by replacing $u^{-1}v$ with $v'u'{}^{-1}$ (such that $uv'=vu'$ is a
relation of the considered \pres), then the involved factor $u^{-1}v$ is associated with a diverging pair of edges in a path
labelled $w_i$ and we complete the graph by closing the open pattern $u^{-1}v$ using right-oriented edges labelled $v'$
and down-oriented edges labelled $u'$, see Fig.~\ref{fig:rev_Uv}.

\begin{figure}[htb]
\begin{picture}(130,100)(70,0)
\put(0,2){\includegraphics{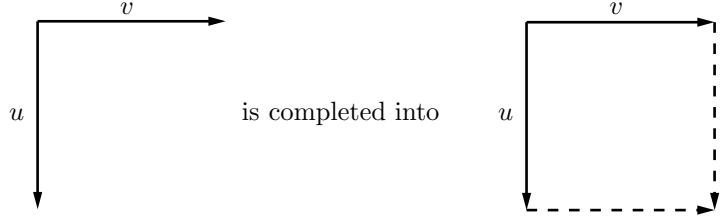}}
\put(-9,38){$u$}
\put(176,38){$u$}
\put(79,38){is completed into}
\put(33,78){$v$}
\put(218,78){$v$}
\end{picture}
\caption{Reversing of $u^{-1}v$ into $v'u'{}^{-1}$.}
\label{fig:rev_Uv}
\end{figure}

The case of the empty word $\ew$, which appears when a factor $u^{-1}u$ is deleted or some relation $uv'=v$ is used, is treated similarly: we introduce $\ew$-labelled edges and use them according to the conventions $\ew^{-1}u\rev u\ew^{-1}$, $u^{-1}\ew \rev \ew u^{-1}$, and $\ew^{-1}\ew \rev \ew\ew^{-1}$.

 The word that is being reversed  along the \redr\ sequence appears on the graph as the top left border: concatenate the
labels of the top left border (with the convention that an arrow crossed backwards contributes  with an exponent $-1$) to
get the word. See Fig.~\ref{fig:rev_bbb} for a simple example of \redr\ graph.


\def\Grid{%
\linethickness{0.3pt}
\multiput(0,0)(10,0){40}{\line(0,1){140}}
\multiput(0,0)(0,10){15}{\line(1,0){400}}
\put(-1,-8){ \tiny 0}
\put(8,-8){ \tiny 10}
\put(18,-8){\tiny 20}
\put(28,-8){\tiny 30}
\put(38,-8){\tiny 40}
\put(48,-8){\tiny 50}
\put(58,-8){\tiny 60}
\put(68,-8){\tiny 70}
\put(78,-8){\tiny 80}
\put(88,-8){\tiny 90}
\put(98,-8){\tiny 100}
\put(108,-8){\tiny 110}
\put(118,-8){\tiny 120}
\put(128,-8){\tiny 130}
\put(158,-8){\tiny 160}
\put(178,-8){\tiny 180}
\put(198,-8){\tiny 200}
\put(218,-8){\tiny 220}
\put(238,-8){\tiny 240}
\put(258,-8){\tiny 260}
\put(278,-8){\tiny 280}
\put(298,-8){\tiny 300}
\put(318,-8){\tiny 320}
\put(338,-8){\tiny 340}
\put(358,-8){\tiny 360}
\put(378,-8){\tiny 380}
\put(-4,9){\tiny 10}
\put(-4,19){\tiny 20}
\put(-4,29){\tiny 30}
\put(-4,39){\tiny 40}
\put(-4,49){\tiny 50}
\put(-4,59){\tiny 60}
\put(-4,69){\tiny 70}
\put(-4,79){\tiny 80}
\put(-4,89){\tiny 90}
\put(-4,99){\tiny 100}
\put(-4,109){\tiny 110}
\put(-4,119){\tiny 120}
\put(-4,129){\tiny 130}
\put(-4,139){\tiny 140}
}

\begin{figure}[htb]
\begin{picture}(200,165)(75,0)
\put(0,2){\includegraphics{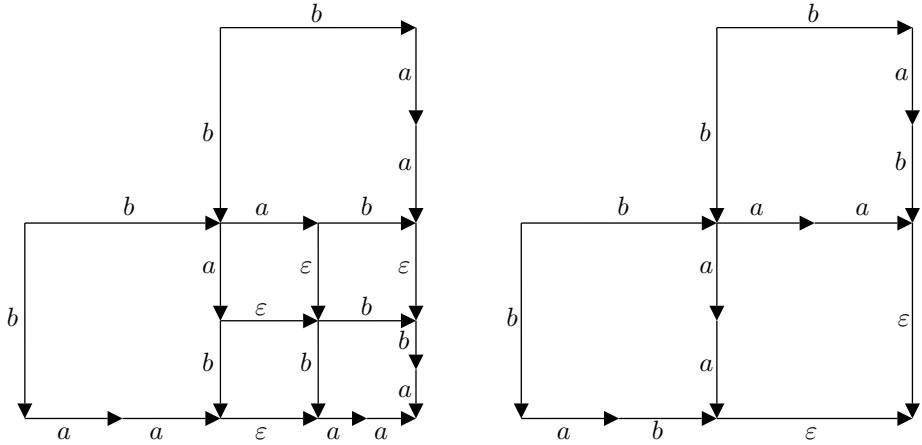}}
\put(15,-3){$a$}
\put(50,-3,){$a$}
\put(117,-3){$a$}
\put(135,-3){$a$}
\put(204,-3){$a$}
\put(144,13){$a$}
\put(70,60){$a$}
\put(90,82){$a$}
\put(144,99){$a$}
\put(144,133){$a$}
\put(258,23){$a$}
\put(258,60){$a$}
\put(277,82){$a$}
\put(317,82){$a$}
\put(332,133){$a$}

\put(240,-3){$b$}
\put(70,23){$b$}
\put(107,23){$b$}
\put(144,31){$b$}
\put(130,44){$b$}
\put(130,82){$b$}
\put(-4,40){$b$}
\put(40,82){$b$}
\put(70,110){$b$}
\put(111,155){$b$}
\put(185,40){$b$}
\put(227,82){$b$}
\put(258,110){$b$}
\put(299,155){$b$}
\put(332,99){$b$}

\put(90,-3){$\ew$}
\put(298,-3){$\ew$}
\put(90,44){$\ew$}
\put(107,60){$\ew$}
\put(144,60){$\ew$}
\put(333,40){$\ew$}

\end{picture}
\caption{Case of the presentation $( a,b;\, bab=ba^2)$, \redr\ of $b^{-1}bb^{-1}b$. Because the relation $bab=ba^2$ is of the type $b\ldots=b\ldots$,  the \redr\ of $b^{-1}bb^{-1}b$ is not deterministic: at each \redr\ step of the type $b^{-1}b$, one can either delete the subword $b^{-1}b$ or replace it with $aba^{-2}$. Among the possible {\redr}s, we only consider those of the form $b^{-1}bb^{-1}b\rev uv^{-1}$ in order to apply Proposition~\ref{prop:word_rev_completion}.}
\label{fig:rev_bbb}
\end{figure}

Now we use \redr\ graphs to study the R-completeness of the \pres\ $( a,b;\, bab=ba^2)$. The word $b^{-1}bb^{-1}b$ reverses  into $a^4{(a^3b)}^{-1}$  and $a^2{(ab)}^{-1}$, and, symmetrically, into  $a^3b{(a^4)}^{-1}$ and $ab{(a^2)}^{-1}$ (Fig.~\ref{fig:rev_bbb}). Hence, according to Algorithm~\ref{algo:word_rev_completeness_criterion}, we have to check ${(ba^4)}^{-1}ba^3b\rev\ew$ and ${(ba^2)}^{-1}bab\rev\ew$ (the symmetric cases are similar). Since $bab=ba^2$ is a relation of the \pres, ${(ba^2)}^{-1}bab\rev\ew$ trivially holds.  

As for ${(ba^4)}^{-1}ba^3b$, Fig.~\ref{fig:rev_AAAABbaaab} shows that it cannot be reversed  to $\ew$. Indeed, there is no relation of the type $a\ldots=b\ldots$ in the \pres, therefore the words ${(a^2b)}^{-1}a^3$ and $a^{-1}b$ cannot be reversed . Hence, by Proposition~\ref{prop:word_rev_completeness_criterion}, the \pres\ is not R-complete, and, applying Algorithm~\ref{algo:word_rev_completion}, we add the relation $ba^4=ba^3b$ to the \pres.


\begin{figure}[htb]
\begin{picture}(200,165)(70,0)
\put(0,2){\includegraphics{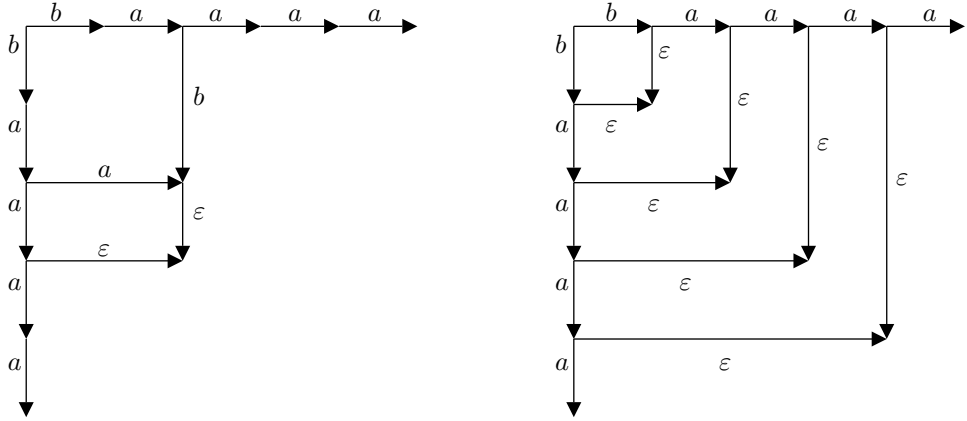}}
\put(42,152){$a$}
\put(72,152){$a$}
\put(102,152){$a$}
\put(132,152){$a$}
\put(-4,110){$a$}
\put(-4,80){$a$}
\put(-4,50){$a$}
\put(-4,20){$a$}
\put(30,93){$a$}
\put(252,152){$a$}
\put(282,152){$a$}
\put(312,152){$a$}
\put(342,152){$a$}
\put(203,110){$a$}
\put(203,80){$a$}
\put(203,50){$a$}
\put(203,20){$a$}

\put(12,152){$b$}
\put(-4,140){$b$}
\put(66,120){$b$}
\put(222,152){$b$}
\put(203,140){$b$}

\put(30,63,){$\ew$}
\put(66,77){$\ew$}
\put(222,110){$\ew$}
\put(242,138){$\ew$}
\put(238,80){$\ew$}
\put(272,120){$\ew$}
\put(250,50){$\ew$}
\put(302,103,){$\ew$}
\put(265,20){$\ew$}
\put(332,90){$\ew$}
  \end{picture}
    \caption{Case of the presentation $( a,b;\, bab=ba^2)$, \redr\ of ${(ba^4)}^{-1}ba^3b$: despite the equivalence $ba^4\eqv ba^3b$, the word ${(ba^4)}^{-1}ba^3b$ does not reverse to \ew. The relation $ba^4 = ba^3b$ is therefore added to the \pres\ so that this equivalence is now provable in terms of \redr.}
  \label{fig:rev_AAAABbaaab}
\end{figure}

We prove now, by induction on $n$, that the \pres\ $(a,b;\, bab= a^2b,\ldots,ba^{2n-1}b=ba^{2n})$ is not R-complete and that Algorithm~\ref{algo:word_rev_completion} leads to adding the relation $ba^{2n+1}b=ba^{2n+2}$. The case $n=1$ was treated above. Assume $n\geq 2$. As illustrated in Fig.~\ref{fig:rev_bbb_to_a2n+2BA2n+1}, $b^{-1}bb^{-1}b$ reverses  to $a^{2n}{(a^{2n-1}b)}^{-1}$ and ${(ba^{2n+2})}^{-1}ba^{2n+1}b$ does not reverse  to $\ew$, so the criterion of Proposition~\ref{prop:word_rev_completeness_criterion} fails, and, according to Algorithm~\ref{algo:word_rev_completion}, we add $ba^{2n+1}b=ba^{2n+2}$.
%

\begin{figure}[htb]
\begin{picture}(200,165)(70,0)
\put(0,2){\includegraphics{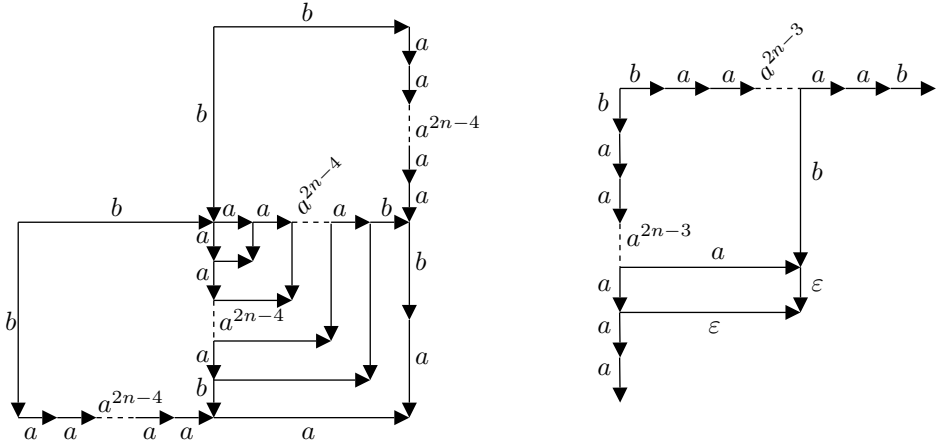}}
\put(5,-3){$a$}
\put(20,-3){$a$}
\put(33,7){$a^{2n-4}$}
\put(50,-3){$a$}
\put(64,-3){$a$}
\put(110,-3){$a$}
\put(70,71){$a$}
\put(70,57){$a$}
\put(79,38){$a^{2n-4}$}
\put(70,26){$a$}
\put(80,81){$a$}
\put(93,81){$a$}
\put(122,81){$a$}
\put(105,81){\rotatebox{45}{$a^{2n-4}$}}
\put(153,144){$a$}
\put(153,130){$a$}
\put(153,111){$a^{2n-4}$}
\put(153,100){$a$}
\put(153,86){$a$}
\put(153,25){$a$}
\put(252,132){$a$}
\put(269,132){$a$}
\put(280,132){\rotatebox{45}{$a^{2n-3}$}}
\put(303,132){$a$}
\put(320,132){$a$}
\put(222,104){$a$}
\put(222,86){$a$}
\put(233,70){$a^{2n-3}$}
\put(222,53){$a$}
\put(222,37){$a$}
\put(222,21){$a$}
\put(265,65){$a$}

\put(70,13){$b$}
\put(-2,38){$b$}
\put(38,81){$b$}
\put(140,81){$b$}
\put(153,61){$b$}
\put(70,117){$b$}
\put(110,155){$b$}
\put(234,132){$b$}
\put(335,132){$b$}
\put(222,121){$b$}
\put(303,95){$b$}

\put(303,53){$\ew$}
\put(264,37){$\ew$}

\end{picture}
\caption{\label{fig:rev_bbb_to_a2n+2BA2n+1}Reversing of the word $b^{-1}bb^{-1}b$ to $a^{2n+2}{(a^{2n+1}b)}^{-1}$, using relations added to the \pres\ $( a,b;\, bab=ba^2)$. The word ${\left(ba^{2n+2}\right)}^{-1}ba^{2n+1}b$ does not reverse to $\ew$ since no relation of the type $b\ldots = a\ldots$ has been added so far.}
\end{figure}

We claim now that the \pres\ $(a,b;\, ba^{2n-1}b=ba^{2n},n\geq 1)$ is R-complete. To prove this, as we have seen above, we reverse  the word $b^{-1}bb^{-1}b$ to all possible words of the form $uv^{-1}$, as required by the completeness criterion (Algorithm~\ref{algo:word_rev_completeness_criterion}). Thus, for any $m<n$, we have the following sequence of {\redr}s:
\[
\begin{array}{ccccl}
b^{-1}bb^{-1}b &\revelem &a^{2n}b^{-1}{(a^{2n-1})}^{-1}b^{-1}b  & \revelem & a^{2n}b^{-1}{(a^{2n-1})}^{-1}a^{2m-1}b{(a^{2m})}^{-1}  \\
                                                         &&     & \revelem & a^{2n}b^{-1}{(a^{2(n-m)})}^{-1}b{(a^{2m})}^{-1}.
\end{array}
\]
From the latter word, unless $m=n$ holds and because there is no relation $a\ldots=b\ldots$, there cannot be a \redr\ of  $b^{-1}bb^{-1}b$  to a word of the type $uv^{-1}$. Suppose we have $m=n$, then for every $p$ we get 
\[
\begin{array}{ccccl}
b^{-1}bb^{-1}b &\revelem& a^{2m}b^{-1}b{(a^{2m})}^{-1} & \revelem & a^{2m}a^{2p}b^{-1}{(a^{2p-1})}^{-1}{(a^{2m})}^{-1} \\
                                                  &&    & \revelem & a^{2(m+p)}b^{-1}{(a^{2(p+m)-1})}^{-1}.
\end{array}
\]                                              
Now, since $ba^{2(p+m)}= ba^{2(p+m)-1}b$ is a relation of the \pres, the criterion for R-completeness is satisfied (Algorithm~\ref{algo:word_rev_completeness_criterion}), which proves that the \pres\ $(a,b;\, ba^{2n-1}b=ba^{2n},n\geq 1)$ is R-complete. So we can conclude with:

\begin{fact}
\label{prop:G-completion_and_R-completion_agree}
Starting with the \pres\ $(a,b;\, bab=ba^2)$, both the G-completion and the R-completion lead to adding the
(infinite)  family of relations $ba^{2n-1}b=ba^{2n}$ with $n \geq 2$.
\end{fact}

Other simple presentations, such as $(a,b;\, a^2b=ba^2, ba^2=a^2b)$ or the Baumslag-Solitar \press\ $(a,b;
ba=a^nb)$, give rise to a similar coincidence phenomenon. So it is natural to raise

\begin{ques}
\label{Q:Main}
Do the G-completion and the R-completion necessarily coincide for every semigroup presentation?---or, at least, on
every presentation in some natural family?
\end{ques}

\section{Divergence results}
\label{sec:complet_diverg}

In this section, we answer Question~\ref{Q:Main} in the negative:

\begin{prop}
\label{prop:main}
There exist finite semigroup presentations for which the G-comp\-letion and the R-completion do not agree. More
precisely, using $\compG$ (resp. $\compR$) for the G-comple\-tion  (resp. the R-completion), there exist finite
semigroup presentations~$\sr$ exhibiting each of the following behaviours:

$\bullet$ type 1: $\compR$ is a proper subset of~$\compG$, with $\compR$ finite and $\compG$ infinite;

$\bullet$ type 1': $\compR$ is a proper subset of~$\compG$, with both $\compR$ and $\compG$ finite;

$\bullet$ type 2: $\compG$ is a proper subset of~$\compR$, with $\compG$ finite and $\compR$ infinite;

$\bullet$ type 3: $\compG$ and $\compR$ are not comparable with respect to inclusion.
\end{prop}

We shall now successively construct examples displaying the various above-mentioned behaviours.

\subsection{Type 1 counter-examples}
\label{s_sec:type_1_counter-examples}

It is relatively easy to find \RG\ \cexs, and we shall exhibit various families.

\begin{prop}
  \label{prop:every-bwb=abw-is-type-RG}
Every presentation 
\begin{equation}
  \label{eq:pres_bwb=abw}
(a,b,c,\ldots;\, bwb=abw), \,\,w\in {\{a,b,c,\ldots\}}^\ast  
\end{equation}
together with any homogeneous lexicographic order with $a$ minimal is a \RG\ \cex.
\end{prop}

\begin{proof}
Let $\Pi_w$ be the \pres\ of~(\ref{eq:pres_bwb=abw}). Each \pres\ $\Pi_w$ is homogeneous (the length is a pseudolength) and has exactly one relation, which is of the type $a\ldots=b\ldots$; by Proposition~\ref{prop:word_rev_completeness_criterion}, the \pres\ $\Pi_w$ is R-complete.

First, consider the case $w=\ew$. Then the composition of $bb=ab$ with itself iteratively leads to the relations $R_m:\, ba^mb=a^{m+1}b$. Now the composition of $R_m$ with $R_n$ reduces to~$0$. Proposition~\ref{prop:minimal_grob_if_no_red_and_no_comp_left} implies that $\{ba^mb=a^{m+1}b; m\geq 0\}$ is a reduced \G\ basis of $\Pi_\ew$. So in this case, the R-completion of $\Pi_\ew$, which is $\Pi_\ew$, is properly included in the G-completion of $\Pi_\ew$, and $\Pi_\ew$ is a \RG\ \cex.

Assume now $w\ne\ew$. Composing $bwb=abw$ with itself gives $bwabw=abw^2b$, which, composed with $bwb=abw$, gives $bwa^2bw=abw^2b^2$. Iterating this\ie\ composing $bwb=abw$ with the result of each previous composition, produces all relations $bwa^mbw=abw^2b^m$ with $m\geq 1$. 

We want to prove that the G-completion $\mathcal{B}$ of the \pres~(\ref{eq:pres_bwb=abw}) is infinite. We have seen that, for each $m\geq 1$, we have $bwa^mbw\eqv abw^2b^m$. It suffices to show that no relation of $\mathcal{B}$ may reduce infinitely many different words $bwa^mbw$. For a contradiction, assume that $(i)$ $u=v$ is a relation of \Bgb\ with $\ell:=|u|$ and $(ii)$ there exists $A\subsetneq \mathbb{N}$ infinite with $\ell \leq \min A$ such that $u=v$ reduces all words $bwa^mbw$ for $m\in A$.

In the sequel, a word $w$ is called \emph{isolated} if it is $\eqv$-equivalent to no other word. For a word $w$, we shall denote by $\sharp_b(w)$ the number of $b$'s in $w$.
 
If we have $\ell\leq 1+|w|$ then $u$ is too short to include $bwb$ or $abw$, and hence $u$ is isolated, contradicting $(i)$. 

\emph{Case 0}: the word $u$ starts at position at least $2+|w|$ and finishes at position at most $2+|w|$ to the end, hence $u$ has the form $a^\ell$. But
$a^\ell$ does not include neither $bwb$ nor $abw$ and is therefore isolated, which contradicts $(i)$. 

\emph{Case 1}: there is a $q$ such that $u$ is a prefix of $bwa^q$. Because $m>\ell$, the word $u$ has the form $bwa^p$, $p\geq 1$. Then it contains no subword $bwb$ because we have $\sharp_b(bwb) > \sharp_b(bwa^p)$. Similarly, $u$ contains no subword $abw$ because $abw\subseteq bwa^p$ implies $abw\subseteq wa^p$, hence $\sharp_b(bw) \leq \sharp_b(w)$, and therefore $u$ is isolated, which contradicts $(i)$.

\emph{Case 2}: the word $u$ starts at position $i$, with $i\geq 2$; hence there is a $q$ such that $u$ is a prefix of $w'a^q$ for some suffix $w'$ of
$w$. We have $\sharp_b(u) < \sharp_b(bwb)$ and $\sharp_b(u)<\sharp_b(abw)$ and so neither $bwb$ nor $abw$ is a subword of $u$; hence $u$ is
isolated, which contradicts $(i)$.

\emph{Case 3}: the word $u$ finishes at position at most $1+|w|$ to the end. Then we have $u=a^pbw'$ with $p\geq 1$ and $w'$ prefix of $w$; since
$\ell > 1+|w|$, we exclude the case where $u$ is a prefix of $w$. Because of the homogeneous lexicographic ordering, a matching word $v$ has the form
$a^pv'$ and then, by cancellativity (see \cite[Prop.~6.1]{dehornoy03:_compl}), $bw$ reduces to $v'$, which is impossible because $bw$ is of length
$1+|w|$ and therefore too short not to be isolated. This contradicts $(i)$.
\end{proof}

Observe in the previous proof that, although the reduced \G\ basis is not computable, we are able to determine that it is necessarily infinite.

A typical instance of Proposition~\ref{prop:every-bwb=abw-is-type-RG} is the standard \pres\ of the braid monoid $B_3^+$.

\begin{example}
\label{ex:B3}
  The \pres\ $(a,b; bab=aba)$ of the braid monoid $B_3^+$, with deglex order induced by $b>a$, satisfies the hypotheses of Proposition~\ref{prop:every-bwb=abw-is-type-RG} and is therefore a \RG\ \cex. An easy computation \cite[Lemma~4.1]{bokut03:_groeb_shirs} gives the reduced \G\ basis 
\[
\{bab=aba\}\cup\{ba^nba=aba^2b^{n-1}; n\ge 2\},
\]
which is in accordance with Proposition~\ref{prop:every-bwb=abw-is-type-RG}.
\end{example}

We shall now give other \cexs. The \pres\ of the braid monoid $B_3^+$ is the first non trivial case of 2-generator Artin \pres, and we can obtain more \RG\ \cexs\ by considering more general Artin \press. 

\begin{prop}
\label{prop:2-gen-Artin-RG-cex}
  Every 2-\gen\ Artin \pres
\[
(a,b ; \underbrace{baba\ldots}_{\text{length }m}=\underbrace{abab\ldots}_{\text{length }m})
\]
 is a \RG\ \cex\ with respect to any homogeneous lexicographic order.
\end{prop}

\begin{proof}
  There are two cases. If the \pres\ has the type $(a,b; {(ba)}^nb={(ab)}^na)$, with $n\geq 1$, then by Proposition~\ref{prop:every-bwb=abw-is-type-RG}, the \pres\ is a \RG\ \cex. 

We may assume that the \pres\ has the form $(a,b; {(ba)}^n={(ab)}^n)$, with $n\geq 1$. Compose ${(ba)}^n={(ab)}^n$ with itself as follows:
\[
\left({(ba)}^n-{(ab)}^n\right)ba-ba\left({(ba)}^n-{(ab)}^n\right) = -{(ab)}^nba+ba{(ab)}^n.
\]
Compose the resulting relation with ${(ba)}^n={(ab)}^n$ to get
\[
\left(ba{(ab)}^n-{(ab)}^nba\right)a-ba^2\left({(ba)}^n-{(ab)}^n\right) = -{(ab)}^nba^2+ba^2{(ab)}^n.
\]
Iterating these compositions yields the family of relations
\[
\left\{{(ba)}^n={(ab)}^n\right\}\cup\left\{ba^p{(ab)}^n={(ab)}^nba^p; p\geq 1\right\}.
\]
By Proposition~\ref{prop:minimal_grob_if_no_red_and_no_comp_left}, it suffices to check that every composition reduces to zero. We compute the composition of $ba^p{(ba)}^n={(ba)}^nba^p$ with $ba^q{(ba)}^n={(ba)}^nba^q$ and leave the other compositions to the reader: 
\begin{align*}
(ba^p{(ab)}^n - & {(ab)}^nba^p)a^q{(ab)}^n-ba^p{(ab)}^{n-1}a (ba^q{(ab)}^n-{(ab)}^nba^q)\hfill\\
&=-{(ab)}^nba^{p+q}{(ab)}^n+ba^p{(ab)}^{n-1}a{(ab)}^nba^q\\
&\red -{(ab)}^n{(ab)}^nba^{p+q} + ba^p{(ab)}^{n-2}aba{(ab)}^nba^q\\
&\red -{(ab)}^{2n}ba^{p+q} + ba^{p+1}{(ab)}^n{(ba)}^{n-1}ba^q\\
&\red -{(ab)}^{2n}ba^{p+q} + {(ab)}^nba^{p+1}{(ba)}^{n-1}ba^q\\
&\red -{(ab)}^{2n}ba^{p+q} + {(ab)}^nba^{p}{(ab)}^{n}a^q\\
&\red -{(ab)}^{2n}ba^{p+q} + {(ab)}^n{(ab)}^{n}ba^pa^q = 0.
\end{align*}
\end{proof}


Another infinite family of \RG\ \cexs\ extending Example~\ref{ex:B3} in an other  direction than Proposition~\ref{prop:2-gen-Artin-RG-cex} is the
family of standard \press\ of braid monoids:

\begin{prop}
\label{prop:braid_monoid_RG}
For $n \geq 3$, the Artin presentation
\begin{equation}
\label{pres:artin_pres_braids}
\left(
	\sigma_1,\ldots, \sigma_{n-1} \ \bigg\vert\   
	\begin{array}{ll} 
		\sigma_i\sigma_j\sigma_i=\sigma_j\sigma_i\sigma_j &
\text{for } |j-i|=1\\
\sigma_i\sigma_j=\sigma_j\sigma_i &\text{for } |j-i|\geq
2\\
		\end{array}
\right)
\end{equation}
of the braid monoid~$B_n^+$ is a \RG\ counter-example.
\end{prop}

\begin{proof}
  It is a standard result, deduced from Proposition~\ref{prop:word_rev_completion}, that the monoids $B_n^+$ are R-complete.

Take any $i\leq n-2$; put $b=\max(\sigma_i,\sigma_{i+1})$ and $a=\min(\sigma_i, \sigma_{i+1})$. As in the case of $B_3^+$, the relation $bab=aba$ of $B_n^+$ leads Algorithm~\ref{algo:grob_completion_monoid} to add all the relations $ba^nba = aba^2b^{n-1}$, with $n\geq 2$. It suffices to prove that these relations are not reduced by the relations of the reduced \G\ basis \Bgb. For a contradiction, suppose there is a relation $u=v$ in $\Bgb -\left(\{bab=aba\}\cup\{ba^nba=aba^2b^{n-1}; n\geq 2\}\right)$ reducing at least one relation $ba^nba=aba^2b^{n-1}$. So we have $u\in{\{a,b\}}^\ast$. Now, $u\eqv v$ implies that there exist words $u_0, u_1, \ldots,u_n$ satisfying
\[
u=u_0\eqvelem u_1 \eqvelem \cdots \eqvelem u_{n-1} \eqvelem u_n=v.
\]
But there is a single relation in the \pres\ of $B_n^+$ involving $a$'s and $b$'s, namely $bab=aba$. Therefore, $u_0\eqvelem u_1$ implies $u_1\in{\{a,b\}}^\ast$, and it follows that $v$ is in ${\{a,b\}}^\ast$ and that the relation $u=v$ holds in $B_3^+$; hence the relations $ba^nba=aba^2b^{n-1}$, with $n\ge 2$, and $bab=aba$ reduce $u=v$ to zero, contradicting the fact that \Bgb\ was reduced.
\end{proof}

\begin{example}
\label{ex:B4}
By Proposition~\ref{prop:braid_monoid_RG}, we know that the standard Artin \pres\ of the braid monoid $B_4^+$  
\[
( a,b,c; bab =aba, ca=ac, cbc=bcb)
\]
and order induced by $c>b>a$ is a counter-example of \RG. Actually, a direct computation shows that its G-completion is:
\begin{eqnarray*}
bab&=&aba,\\
cbc&=&bcb,\\
ca&=&ac,\\
  ba^nba &=&aba^2b^{n-1}, n\ge 2,\\
cb^ncb &=&bcb^2c^{n-1}, n\ge 2,\\
cba^nc& =& bcba^n, n\ge 1,\\
cba^nb^pcb&=&bcba^nbc^p, n\ge 2, p\ge 1,\\
cb^{n_1}a^{n_2}b^{n_3}\ldots b^{n_k}cba&=&bcb^2ac^{n_1-1}b^{n_2}c^{n_3}\ldots c^{n_k},
\end{eqnarray*}
with $k\geq 2$, and the $n_i$'s are positive integers satisfying $n_2,n_3,\ldots,n_{k-1}\geq 2$, with the additional constraints: $n_1\geq 2$ if $k=2$ or $k=3$ holds, and $n_k\geq 2$ if $k$ is odd.
\end{example}

\begin{rem}
\label{rem:bokut_braid_monoid}
 Bokut \emph{et al.} \cite[Th.~4.2]{bokut03:_groeb_shirs} give \G\ bases for every braid monoid~$B_n^+$, with $n\geq 3$. The latter coincide with
the ones computed in Examples~\ref{ex:B3} and \ref{ex:B4}. Although almost explicit, these bases are neither reduced nor minimal for the cases $n\geq
5$ and therefore do not allow to conclude that the \pres\ of~(\ref{pres:artin_pres_braids}) is a \RG\ counter-example, contrary to
Proposition~\ref{prop:braid_monoid_RG}.
\end{rem}

So far, we considered type~1 counter-examples. We conclude with what was called a type~1' counter-example, namely one where the R-completion
is a proper subset of the G-completion and both are finite.

\begin{prop}
  For every $n\geq 1$ and $p\geq 1$, the \pres
  \begin{equation}
    \label{eq:ab^na=bb}
(a,b; {(ab)}^na = b^p)    
  \end{equation}
together with the homogeneous lexicographic ordering induced by $b>a$ is a type~1' \cex.
\end{prop}

\begin{proof}
A 2-\gen\ \pres\ $(a,b;\Rp)$ whose sole relation has the form $a\ldots = b\ldots$ is R-complete \cite[Prop.~6.4]{dehornoy02:_group_garsid}, and hence the \pres~(\ref{eq:ab^na=bb}) is R-complete. Then, using Proposition~\ref{prop:minimal_grob_if_no_red_and_no_comp_left}, one checks that the set
\[
\Bgb=\{{(ab)}^na=b^p, b^{p+1}a=ab^{p+1}\}
\]
is the reduced \G\ basis of the \pres~(\ref{eq:ab^na=bb}).
\end{proof}

\subsection{Type 2 counter-examples}
\label{s_sec:type_2_counter-examples}

In this section, we give examples of \press\ whose G-completion is properly included in their R-completion. 

\begin{lemma}
\label{lemma:bw=b_implique_R-completion_infinie}
  Assume that $(a,b;\,\Rp)$ is a \pres\ with no relation $a\ldots=a\ldots$ or $b\ldots=b\ldots$. Then, for each nonempty word $w$ on $\{a,b\}$, the R-completion of $(a,b; \Rp, bw=b)$ includes $\{bw^n=b; n\in\mathbb{N}\}$.
\end{lemma}

\begin{proof}
  The relation $bw=b$ implies $bw^n\eqv b$, with $n\geq 1$. We prove by induction on $n$, that ${(bw^n)}^{-1}b$ cannot be reversed  to $\ew$ even if all the relations $bw^m=b$, $m<n$, have been added to the \pres. Since ${(bw^2)}^{-1}b$ reverses  to $w^{-1}$, we can assume $n>2$. By hypothesis, there is no relation $s\ldots=s\ldots$ in \Rp\ and hence, the only {\redr}s of ${(bw^n)}^{-1}b$ are, for every $p$ and $m$ satisfying $p<m<n$, ${(bw^n)}^{-1}b\rev w^{-m}w^p \rev {(w^{m-p})}^{-1}$; this completes the induction. 
\end{proof}

\begin{prop}
\label{prop:bw=b-GR-cex}
  Under hypotheses of Lemma~\ref{lemma:bw=b_implique_R-completion_infinie}, every G-complete \pres\ $(a,b; bw=b)$ is a \cex\ of \GR.
\end{prop}

\begin{proof}
  By Lemma~\ref{lemma:bw=b_implique_R-completion_infinie}, the R-completion of $(a,b; bw=b)$ contains the set $\{bw^n=b; n\in\mathbb{N}\}$ which, in turn, contains the G-completion, namely $\{bw=b\}$.
\end{proof}

\begin{example}
  The simpliest instance of Proposition~\ref{prop:bw=b-GR-cex} is the \pres\ $(a,b;ba=b)$ whose \G\ basis consists in the sole relation $ba=b$ and whose R-completion is $\{ba^n=b; n\in\mathbb{N}\}$.
\end{example}

The next result is another application of Lemma~\ref{lemma:bw=b_implique_R-completion_infinie} differing from Proposition~\ref{prop:bw=b-GR-cex}
in that the set $\Rp$ is nonempty.

\begin{prop}
For every $n$, $q$, $p$ satisfying $n+q>p$, the \pres\
\[
(a,b;\,a^nb^q=b^p, ba=b)
\]
with order induced by $b>a$ is a \cex\ of \GR.
\end{prop}

\begin{proof}
  We first compute the G-completion. There is only one composition available:
  \begin{eqnarray*}
(a^nb^q-b^p)a-a^nb^{q-1}(ba-b) & = & -b^pa+a^nb^q \\
& \red & -b^p + a^nb^q  \quad \red \quad  0.
  \end{eqnarray*}
Hence the \pres\ $(a,b;\,a^nb^q=b^p, ba=b)$ is G-complete. Lemma~\ref{lemma:bw=b_implique_R-completion_infinie} yields the result.
\end{proof}

\begin{rem}
  It is natural, as in Section~\ref{s_sec:type_1_counter-examples}, to define a type~2' \pres\ to be a \pres\ \sr\ satisfying  $\compG \subsetneq \compR$ and  $\vert \compR \vert < \infty$. However, contrary to the type~1', we could not find a \pres\ of type~2' so far. 
Most of the difficulty resides in the computation of R-completions. Indeed, during the completing process, the \redr\ operation often becomes, if it was
not already, non-deterministic: if, at some point,  we have two relations $s\ldots=s\ldots$ at our disposal and $s^{-1}s$ is to be reversed, then we can
reverse in two different ways, leading to two different words.
\end{rem}

\subsection{Type~3 counter-examples}
\label{sec:mixed-type-counter-examples}

We conclude with examples where both completions are incomparable with respect to inclusion. First, we observe that
mixing examples of what were called types~1 and~2 immediately leads to examples of \mixed. But, then, we show that
less artificial examples exist, namely the standard Heisenberg presentation.

We denote by $X_1\sqcup X_2$ the disjoint union of two sets $X_1$ and $X_2$ (that is, $X_1\sqcup X_2=X_1\times \{1\} \cup X_2\times \{2\})$.

\begin{defi}
  Let $(\Spalt_1,\Rp_1)$ and $(\Spalt_2,\Rp_2)$ be two \press. The \emph{direct product} $(\Spalt_1,\Rp_1)\times(\Spalt_2,\Rp_2)$ is the \pres\ $(\Spalt_1\sqcup\Spalt_2,\Rp_1\sqcup\Rp_2\sqcup \Rp)$ with $\Rp=\left\{ s_1s_2=s_2s_1; s_1\in\Spalt_1, s_2\in\Spalt_2\right\}$.
\end{defi}

In the sequel, the orderings considered on the direct product of two ordered \press\ will be the deglex order where the letters are ordered as follows: the orders on  $\Spalt_1$ and $\Spalt_2$ are preserved and we put $\max{\Spalt_1}<\min\Spalt_2$.

\begin{lemma}
\label{lem:direct_product-pres}
Let $\mathcal{P}$ be the direct product of the semigroup \press\ $(\Spalt_1,\Rp_1)$ and $(\Spalt_2,\Rp_2)$. Then, using above notation, the reduced \G\ basis of $\mathcal{P}$ is $\widehat{\mathcal{R}_{1}}{}^G\sqcup \widehat{\mathcal{R}_{2}}{}^G\sqcup \Rp$ and its R-completed set of relations is $\widehat{\mathcal{R}_{1}}{}^R\sqcup \widehat{\mathcal{R}_{2}}{}^R\sqcup \Rp$.
\end{lemma}

\begin{proof}
  To prove that $\widehat{\mathcal{R}_{1}}{}^G\sqcup \widehat{\mathcal{R}_{2}}{}^G\sqcup \Rp$ is the reduced \G\ basis of $\mathcal{P}$, by Proposition~\ref{prop:minimal_grob_if_no_red_and_no_comp_left}, it suffices to check that all compositions reduce to zero. There are no compositions left neither in $\widehat{\mathcal{R}_{1}}{}^G$ nor in $\widehat{\mathcal{R}_{2}}{}^G$ since both sets are reduced. It is obvious that $\Rp$ contains no composition either. The only possible compositions not reducing to zero must therefore involve relations of two different sets among $\widehat{\mathcal{R}_{1}}{}^G$, $\widehat{\mathcal{R}_{2}}{}^G$ and \Rp. Since $\Spalt_1$ and $\Spalt_2$ have no intersection, there are no possible compositions between $\widehat{\mathcal{R}_{1}}{}^G$ and $\widehat{\mathcal{R}_{2}}{}^G$. Because $s_2>s_1$ holds for each $s_2$ in $\Spalt_2$ and $s_1$ in $\Spalt_1$, the first letter of the leading word of every relation of \Rp\ lies in $\Spalt_2$ and the last letter lies in $\Spalt_1$. We leave the reader check that every composition of relations of $\widehat{\mathcal{R}_{2}}{}^G$ and $\Rp$ reduces to zero, the case involving $\widehat{\mathcal{R}_{1}}{}^G$ being similar.

To prove that $(\Spalt_{1}\sqcup\Spalt_{2}; \widehat{\mathcal{R}_{1}}{}^R\sqcup \widehat{\mathcal{R}_{2}}{}^R\sqcup \Rp)$ is R-complete, it suffices to check that every two equivalent words can be proven so by \redr. Now, if $u$ and $v$ are equivalent, then we have $u\eqv u_{1}u_{2}$ and $v\eqv v_{1}v_{2}$, with $u_{1}, v_{1}$ in $\Spalt_{1}^*$ and $u_{2}, v_{2}$ in $\Spalt_{2}^*$, satisfying $u_{1}\eqv v_{1}$ and $u_{2}\eqv v_{2}$. These latter equivalences are provable by \redr\ (since we have all the relations of $\widehat{\mathcal{R}_{1}}{}^R$ and $\widehat{\mathcal{R}_{2}}{}^R$) and one can check that finding a \redr\ of $u^{-1}v$ to $\ew$ amounts to finding {\redr}s of $u_{1}^{-1}v_{1}$ and $u_{2}^{-1}v_{2}$ to $\ew$.
\end{proof}

The next result gives a method to get \mixed\ \cexs\ starting from \cexs\ of \RG\ and \GR. 

\begin{prop}
Assume that $(\Spalt_1,\Rp_1)$ is a \cex\ of \RG, and that
$(\Spalt_2,\Rp_2)$  is a \cex\ of \GR. Then
$(\Spalt_1,\Rp_1)\times(\Spalt_2,\Rp_2)$ is a \cex\ of
\mixed.
\end{prop}

\begin{proof}
  By hypothesis, the sets $\widehat{\mathcal{R}_{1}}{}^R\sqcup \widehat{\mathcal{R}_{2}}{}^R$ and $\widehat{\mathcal{R}_{1}}{}^G\sqcup \widehat{\mathcal{R}_{2}}{}^G$ are not comparable. Thus, by Lemma~\ref{lem:direct_product-pres}, the \G- and R-completions of $(\Spalt_1,\Rp_1)\times (\Spalt_2,\Rp_2)$ are not comparable.
\end{proof}

The latter proposition gave a way to build \mixed\ \cexs\ as direct products of type~1 and type~2 \cexs. There are however less trivial \press\ of
\mixed\ not arising as direct products.

\begin{prop}
  \label{prop:heisenberg-mixed-type}
When equipped with the order $c>b>a$, the \pres\
\[
(a,b,c; ab=bac, ac=ca, bc=cb)
\]
of the Heisenberg semigroup is a \cex\ of type \mixed.
\end{prop}

\begin{proof}
  Following \cite[Ex.~5.4]{dehornoy03:_compl}, the R-completed \pres\ is $(a,b,c; ab=bac, ac=ca, bc=cb, cba=ab)$. Using Algorithm~\ref{algo:grob_completion_monoid}, we find for the G-completion the set
  \begin{eqnarray*}
    \{cb=bc, ca=ac\} &\cup & \{ba^{n+1}c=aba^{n}; n\geq 0\} \\
    & \cup & \{ ba^{2n}b=a^nb^2a^n; n\geq 1\} \cup \{ ba^{2n+1}b=a^nbaba^n; n\geq 1\}.
  \end{eqnarray*}
We first notice that the G-completion is infinite and therefore Heisenberg \pres\ is neither of type1', nor type~2, nor type~2'.
We see that the relation $ab=bac$ of the R-completion is not in the G-completion and thus Heisenberg \pres\ is not a \RG\ \cex. 
\end{proof}

\section{Cancellativity}
\label{sec:cancellativity}
In this section, we compare R-complete and G-complete presentations in terms
of cancellativity of the associated monoids. Here also, the two notions of
complete presentations lead to divergent results: in the case of an
R-complete presentation, left cancellativity can be read from the presentation
directly, while no such result exists for a G-complete presentation.

\subsection{Reading cancellativity off a complete \pres}
\label{s_sec:reading_properties_off_complete_pres}

One of the nice features of an R-complete \pres\ is that one can very easily 
establish the possible left cancellativity property for the associated monoid by
only inspecting the relations. 

\begin{prop}[\protect{\cite[Prop. 6.1]{dehornoy03:_compl}}]
\label{prop:cancellativity_equivalent_su=sv...}
Assume that \sr\ is an R-complete \pres. Then the monoid \monoid\ admits left cancellation if and only if $u^{-1}v\rev \ew$ holds for every relation of the form $su=sv$ in \Rp\ with $s\in\Spalt$.  
\end{prop}

In particular, we get that proving that a monoid is not left cancellative amounts 
to finding a relation $su=sv$ in \Rp\ for which $u^{-1}v\rev \ew$ does not
hold\ie\ the \pres\ being R-complete, a relation $su=sv$ for which $u\eqv v$
does not hold. This means that if there is an obstruction to cancellativity, then
it appears in the relations of the \pres, as soon as it is R-complete. 

When we consider G-complete presentations  instead, the criterion of
Proposition~\ref{prop:cancellativity_equivalent_su=sv...} remains necessary,
but it is no longer sufficient.

\begin{prop}
\label{lem:necess_condition_cancell}
 $(i)$  Assume \sr\ is a reduced G-complete \pres.  If \Rp\ contains a relation of
the type $su=sv$ with $u,v$ nonempty words of
$\Spalt^\ast$, then the monoid \monoid\ is not left cancellative. 

$(ii)$ There exists a G-complete presentation \sr\ such that $\Rp$ contains  no
relation of the type $su=sv$ with $u,v$ nonempty words of $\Spalt^\ast$, and
nevertheless the monoid \monoid\ does not admit left cancellation.
\end{prop}

\begin{proof}
$(i)$ Suppose the monoid \monoid\ is left cancellative. Thus we have $u\eqv v$ and $u=v$ is not a relation of \Rp, otherwise the \pres\ \sr\ would not be
reduced. We prove that this is not possible.

Since we fixed an order compatible with the concatenation in the monoid, the inequality $su > sv$ implies $u>v$. This latter inequality combined with the equivalence $u\eqv v$ means that $u$ or $v$ can be reduced to its normal form. Hence suppose there is a relation $w=w'$ in \Rp\ with $w$ a subword of $u$. This means that the relation $su=sv$ was not reduced, which contradicts the hypothesis. The same applies to~$v$. Thus there is no relation $w=w'$ with $w$ a subword of $u$ or $v$. Therefore $u$ and $v$ must be reduced, which contradicts $u>v$.

$(ii)$   Take the monoid presented by $(a,b,c;\, ca = ba, cb=ba)$. With the 
homogeneous lexicographic order induced by $c>b>a$, this \pres\ is
G-complete. From $ca=ba$ and $cb=ba$ we get $ca\eqv cb$. Now, $a\eqv b$
does not hold and hence the monoid is not left cancellative.
\end{proof}

Proposition~\ref{lem:necess_condition_cancell} establishes that  for a \pres\ to
give rise to a cancellative monoid, there has to be no relation $s\ldots =
s\ldots$ and conversely, that even without relation of the type
$s\ldots=s\ldots$, there exist G-complete \press\ associated to non
cancellative monoids. In the proof, the relations of the \pres\ $(a,b,c;\, ca = ba,
cb=ba)$ suggest that cancellativity might be linked to the particular
presentations possessing two relations $s u = w$ and $s v= w$ with $u \not=
v$ and  $w$ not starting with an $s$. This is not the case:

\begin{prop}
 There exist a G-complete presentation $\sr$ such that the
mon\-oid~$\monoid$ is not left cancellative yet $\Rp$ contains no pair of
relations $s u = w$, $s v= w$ with $u \not= v$ and  $w$ not starting with~$s$.
\end{prop}

 \begin{proof}
   Consider the \pres
\[
\sr=(a,b,c,r,s,t;\, sba=tca, cab =bb, tbb=rcb, sa =rc)
\]
and the homogeneous lexicographic order induced by $a<b<c<r<t<s$. By Prop.~\ref{prop:minimal_grob_if_no_red_and_no_comp_left} this \pres\ is G-complete. In the monoid \monoid, we have $sab\eqv rcb$ and $sbab\eqv rcb$ and hence $sbab\eqv sab$. If the monoid is left cancellative, then we must have $bab \eqv ab$. Now $bab$ and $ab$ are both reduced and therefore not equivalent.
 \end{proof}

\subsection{Infinite completions}
\label{s_sec:infinite-completions}
We have considered in Section~\ref{s_sec:type_1_counter-examples} many 
infinite G-completions of \press\ associated to cancellative monoids.
Contrastingly, all above-mentioned examples involving
an infinite R-completion turn out to be associated with monoids that are not left
cancellative, and one could wonder whether this situation necessarily occurs.
Actually, it is not the case: 

\begin{prop}
\label{prop:exists_infinte_completion_left-cancel}
  There exists a finite \pres\ admitting an infinite R-comple\-tion yet the
associated monoid is left and right cancellative.
\end{prop}

\begin{proof}
  Consider the \pres\
  \begin{equation}
    \label{eq:infinte_R_completion_cancellative}
    (a,b,c,d; \, ab=bac, bc=cbd, da=ad, bd=db, dc=cd).
  \end{equation}
Denote by $\sharp_{a<b}(u)$ (resp. $\sharp_{b<a}(u)$)  the number of pairs $(i,j)$ with $i<j$ such that the $i$th letter of $u$ is an $a$ (resp. $b$) and the $j$th letter is a $b$ (resp. $a$). We define $\lambda$ on the set of words as follows:
\[
\lambda(u) = |u| + 2 \sharp_{a<b}(u) + \sharp_{b<a}(u).
\]
One checks on the relations of (\ref{eq:infinte_R_completion_cancellative}) that $\lambda$ is a pseudolength. Applying Algorithm~\ref{algo:word_rev_completion}, we find that the R-completion for the \pres\ of~(\ref{eq:infinte_R_completion_cancellative}) is the set
\[
\{ab=bac, bc=cbd, da=ad, bd=db, dc=cd\} \cup \{ b{(ac)}^nc = da^ncb ; n\geq 0\}.
\]
Finally, Proposition~\ref{prop:cancellativity_equivalent_su=sv...} and the fact
that there is no relation $s\ldots = s\ldots$
in~\eqref{eq:infinte_R_completion_cancellative} imply that the
associated monoid is left cancellative.

As for right cancellativity, we appeal to left 
reversing~\cite{dehornoy03:_compl}, a notion symmetric to that of (right)
reversing. As the (right) R-completeness involved in
Proposition~\ref{prop:cancellativity_equivalent_su=sv...} leads to left
cancellativity, left R-completeness leads to right cancellativity. So it suffices to
prove that the \pres~\eqref{eq:infinte_R_completion_cancellative} is left
R-complete; this is similar to proving its right R-completeness and hence we
omit it.
\end{proof}







\bibliography{biblio.bib}
\bibliographystyle{plain.bst} 

\end{document}